%% file: main.tex
\documentclass{birkjour}
 \newtheorem{thm}{Theorem}[section]
 \newtheorem{cor}[thm]{Corollary}
 \newtheorem{lem}[thm]{Lemma}
 \newtheorem{prop}[thm]{Proposition}
 \theoremstyle{definition}
 \newtheorem{defn}[thm]{Definition}
 \theoremstyle{remark}
 \newtheorem{rem}[thm]{Remark}
 \newtheorem{ex}{Example}
 \numberwithin{equation}{section}

 \usepackage{cite}
\usepackage{hyperref}
\hypersetup{
	colorlinks,
	linkcolor = {blue},
	citecolor = {blue},
	filecolor = {blue},
	urlcolor = {blue} 
}

\definecolor{r}{rgb}{255, 0, 0}

\begin{document}

\title[Inscribed Sphere and Apollonius Centers]
 {On the Inscribed Sphere and Concurrent Lines through the Centers of Apollonius Spheres in $\mathbb{R}^n$}

%----------Author 1
\author[Miłosz Płatek]{Miłosz Płatek}

\address{%
Independent Researcher\\
Kraków, Poland}

\email{milosz@platek.org}

\subjclass{51B25; 51B15; 51M15; 51M04}

\keywords{Apollonius problem, Lie sphere geometry, Euclidean $n$--space, tangent spheres, concurrency of lines through centers, Soddy line}

\date{January 1, 2004}

\begin{abstract}
The Apollonius problem asks for a sphere tangent to $n+1$ given spheres or hyperplanes in $\mathbb{R}^n$. This problem has been widely studied for an isolated configuration of $n+1$ spheres. In this paper, we study relations among the solutions of the Apollonius problem arising from a common family of spheres within the framework of Lie sphere geometry. More precisely, we consider a configuration of $n+2$ spheres in $\mathbb{R}^n$ and the solutions of the Apollonius problem corresponding to all its subsets of size $n+1$. The first main result concerns lines passing through the centers of pairs of solutions to the Apollonius problem. We prove that all these lines intersect at a single point $P_X$. We then introduce a two--step construction of further Apollonius spheres and show that the lines determined by their centers also pass through $P_X$. This yields numerous applications in two and three dimensions and, at the same time, automatically extends them to $\mathbb{R}^n$. The second main result is an $n$--dimensional generalization of K. Morita's three-dimensional theorem on the inscribed sphere in a configuration of mutually tangent spheres. We show that Morita's theorem is a special case of our result for an arbitrary configuration of $n+2$ spheres in $\mathbb{R}^n$, not necessarily mutually tangent. Moreover, we connect this result with the preceding ones by proving that the center of the corresponding inscribed sphere is again the point $P_X$.
\end{abstract}

%%% ----------------------------------------------------------------------
\maketitle
%%% ----------------------------------------------------------------------
%\tableofcontents
\section{Introduction}

The Apollonius problem asks for a sphere tangent to $n+1$ given spheres or hyperplanes in $\mathbb{R}^n$. Various aspects of this problem have been investigated.

Configurations of mutually tangent circles were studied by Soddy, a Nobel Prize laureate in Chemistry, and famously presented in his poem \textit{The Kiss Precise}, published in Nature~\cite{kissprecise}. He considered such configurations in two dimensions and described relations between the radii of mutually tangent circles and spheres, now known as Descartes' theorem, first stated by Descartes in 1643 in letters to Princess Elisabeth of the Palatinate~\cite{history}.

The configuration of $n+1$ spheres in $n$ dimensions has been studied extensively in the literature~\cite{2001,2004,2015}. In particular, the number of real solutions to the Apollonius problem and the properties of mutually tangent spheres have been investigated, among other aspects~\cite{tobefilled,paluszny}. The planar and three-dimensional cases, as well as various degenerations of the configuration, have also been studied~\cite{nature, ijcdm, degeneration, sodyline, general}.

These configurations have not been considered solely in a mathematical context. They also arise in biomedical research, in particular in the study of the medial axis of the space surrounding a molecule and in Voronoi diagrams used, for example, in protein structure analysis~\cite{biologia,voronoi}. In the case of Voronoi diagrams, one seeks the centers of spheres that are solutions to the Apollonius problem for many configurations of quadruples of spheres in three-dimensional space~\cite{voronoi}.

When considering the Apollonius problem for many subsets of a given set of spheres, the resulting configurations are not independent, since some of them share common spheres. Consequently, solutions associated with certain subsets may help determine solutions and properties of others.

In this paper, we establish such relations between Apollonius configurations. More precisely, we study a set of $n+2$ spheres in $\mathbb{R}^n$, which determine $n+2$ Apollonius configurations, each of which shares $n$ spheres with every other such configuration.

In Section~\ref{prel}, we introduce the notions needed throughout the paper. Since the proofs rely heavily on Lie sphere geometry, which was also used to study the Apollonius problem in~\cite{2001,2004,2015, knight}, the basic facts from this theory are presented there.

In Section~\ref{sec3}, we present the main results of the paper. We study the above-mentioned configuration of $n+2$ spheres in $\mathbb{R}^n$. First, in Subsection~\ref{sec31}, we examine the relations between the centers of the solutions to the Apollonius problem in this configuration. We formulate and prove Proposition~\ref{prop31}, which shows that the lines through the centers of the Apollonius spheres meet at a single point $P_X$. Next, we present the first main result: a two-step construction of Apollonius spheres, which are spheres tangent to previously obtained Apollonius spheres, and again show that the lines passing through their centers pass through the point $P_X$ (Theorem~\ref{main1}). There are $n\cdot2^n$ such constructions, and thus we obtain exactly that many lines passing through the point $P_X$. The main value of the results presented in this section lies in their simple and broad applicability to the geometric configurations discussed in Section~\ref{sec4}. Subsequently, in Subsection~\ref{sec32}, we prove the second main result of this paper, which generalizes Morita's Theorem~\ref{mort} on the existence of a sphere tangent to given planes~\cite{general}. Morita established this result for the case of pairwise externally tangent spheres in dimensions two and three. We prove this theorem in $n$ dimensions for an arbitrary configuration of spheres, not necessarily pairwise tangent, showing that Morita's Theorem is a special case of the general relationship between solutions to the Apollonius problem. We present this result as Theorem~\ref{main2}, constituting the second main result of the paper. This leads to a general theorem on the existence of a sphere inscribed among hyperspheres tangent to the solutions of the Apollonius problem. We prove that its center is the point $P_X$, thereby connecting this theorem with the preceding results. Although the proofs rely heavily on Lie sphere geometry, the statements of the theorems are formulated in classical Euclidean space and do not employ the language of Lie geometry. As a result, these findings admit numerous applications without the need to invoke the underlying linear algebra.

Finally, in Section~\ref{sec4}, we present a wide range of applications of these results arising from their degenerations. We show applications of these theorems to problems of varying difficulty, ranging from rather intricate ones, such as concurrencies of lines on the Soddy line (Corollary~\ref{cor45soddy}), to elementary ones, such as the existence of the circumcenter or incenter (Example~\ref{check1} and Example~\ref{check2}). The results presented there are formulated in the plane, as this makes them easiest to visualize. Their proofs follow directly from our theorem and immediately yield $n$--dimensional generalizations of these statements, which, to the best of our knowledge, have not previously been stated in this generality for the corollaries presented there. This section also contains new results: we introduce a new theorem about outer Apollonius circles in a triangle (Corollary~\ref{outerapol}), which is a direct corollary of our main result, and introduce the notion of the generalized Soddy line (Definition~\ref{gensody}), which fits naturally into our framework.

\section{Preliminaries}\label{prel}
This section provides the necessary background on Lie sphere geometry used in the paper. Our notation and basic facts follow the standard framework of Lie sphere geometry; see, for example,~\cite{bb,2015}. We present the material from first principles in order to make the paper as self-contained and accessible as possible.
\subsection{Notation}

Let $V=\mathbb{R}^{n+1,2}$ be a real vector space equipped with a nondegenerate symmetric bilinear form $(.|.)$ of signature $(n+1,2)$ defined by
\begin{equation}
(X|Y)=-x_1y_1+\sum_{i=2}^{n+1}x_iy_i+x_{n+2}y_{n+3}+x_{n+3}y_{n+2},
\end{equation}
where $X=(x_1,\,\dots,\,x_{n+3})$ and $Y=(y_1,\,\dots,\,y_{n+3})$.
We denote by $\mathbb{P}(V)$ the projective space of equivalence classes of nonzero vectors in $V$, denoted by $[X]=\mathbb{R}^*X$.

In $\mathbb{R}^n$, directed spheres, hyperplanes and point spheres (spheres with $r=0$) are treated uniformly in Lie sphere geometry, and we call them \textit{cycles} (for $n=2$ this includes circles, lines and points).

\begin{defn}\label{eqcycles}
   For fixed $v,\,\,$$c_1,\,c_2,\,\dots,\,c_n,\,w\in\mathbb{R}$ an equation of an $n$-dimensional cycle $S$ is given by
    \begin{equation}
        S=\{x=(x_1,\,\dots,\,x_n)\in\mathbb{R}^n|\ v(x_1^2+\cdots+x_n^2)-2(c_1x_1+\cdots+c_nx_n)-2w=0\}.
    \end{equation}
    Here $v\neq0$ corresponds to a sphere and $v=0$ corresponds to a hyperplane. It satisfies
    \begin{equation}
        2wv=r^2-(c_1^2+c_2^2+\cdots+c_n^2),
    \end{equation}
    where $r$ is the radius of the cycle.
    The homogeneous center is $c=[c_1:\cdots: c_n: v]$, where $v=0$ represents a point at infinity.
\end{defn}

\begin{rem}
    Note that the radius may be negative. Its sign is determined by the orientation of the cycle. For more on directed circles and their tangency, see~\cite{directed}.
\end{rem}

\begin{defn}[Cycle representation]\label{representation}
For a cycle $S$, we associate the projective class $[X] \in \mathbb{P}(V)$ of a nonzero vector
\begin{equation}
X=(r,\, c_1,\,\dots,\, c_n,\, v,\, w)\in V.
\end{equation}
\end{defn}

\begin{lem}[Cycles]
A projective point $[X]\in \mathbb{P}(V)$ represents a cycle if and only if $(X|X)=0$.
\end{lem}

\begin{lem}[Tangency]
Two cycles $[X]$ and $[Y]$ are tangent if and only if $(X|Y)=0$.
\end{lem}

\begin{defn}[Lie quadric]
The Lie quadric is defined as
\begin{equation}
Q=\{[X]\in \mathbb{P}(V)|(X|X)=0\}.
\end{equation}
\end{defn}

\subsection{Lie sphere transformations}
Lie sphere transformations are projective transformations preserving the Lie quadric $Q$. In Definition~\ref{def27}, we introduce one such transformation that will be used in the proof of Theorem~\ref{main1}.

\begin{defn}\label{proorto}
    For projective points $[X_1], \,[X_2], \,\dots,\, [X_k] \in \mathbb{P}(V)$ we define
    \begin{equation}
    \begin{split}
        \langle [X_1],\,[X_2],\,\dots,\, [X_k]\rangle &= \ \mathbb{P}(\operatorname{span}(X_1,\,X_2,\,\dots,\,X_k)),\\
        \langle [X_1],\,[X_2],\,\dots,\, [X_k]\rangle^{\perp} &= \{[X]\in \mathbb{P}(V)| \forall_{1 \leq i\leq k}\colon (X_i|X)=0\}.
    \end{split}
\end{equation}
\end{defn}

\begin{lem}[Decomposition with respect to $P$]\label{decomposition}
Let $P\in V$ satisfy $(P|P)\neq 0$. Then
\begin{equation}
   V = \langle P\rangle^\perp\oplus\langle P\rangle, 
\end{equation}
In particular, every $X\in V$ can be written uniquely as
\begin{equation}\label{division}
        X=X_\perp+\alpha P,
    \qquad
    X_\perp\in \langle P\rangle^\perp,
    \qquad
    \alpha=\frac{(X|P)}{(P|P)}.
\end{equation}
\end{lem}
\begin{defn}[Reflection in P]\label{def27}
Let $[P]\in \mathbb{P}(V)$ satisfy $(P|P)\neq 0$. The reflection (orthogonal symmetry) with respect to the hyperplane $\langle [P]\rangle ^\perp$ is the projective transformation $R_P\colon (\mathbb{P}(V)\cup V) \to (\mathbb{P}(V)\cup V)$ given by
\begin{equation}
\begin{split}
    R_P(X) &= X-2\frac{(X|P)}{(P|P)}P,\\
    R_P([X]) &= [R_P(X)].
\end{split}
\end{equation}
\end{defn}
\begin{rem}
    Since $R_P([\lambda X]) = [R_P(\lambda X)] = [\lambda R_P(X)] = [ R_P(X)]$, we have that above definition is well defined.
\end{rem}
\begin{lem}[Properties of $R_P$]\label{properties}
For all $[X],\,[Y]\in \mathbb{P}(V)$ the following hold:
\begin{itemize}
\item $R_P$ is an involution, i.e. $R_P\circ R_P=\mathrm{id}$,
\item $(R_P(X)|R_P(Y))=(X|Y)$.
\end{itemize}
\end{lem}
\begin{rem}
Since $(R_P(X)|R_P(X))=(X|X)$ for all $X\in V$, we have 
\begin{equation}
    [X]\in Q\iff R_P([X])\in Q.
\end{equation}
Hence $R_P$ induces a projective transformation of $\mathbb{P}(V)$ preserving the Lie quadric $Q$, i.e. a Lie sphere transformation.
\end{rem}

\subsection{Apollonius cycles}

\begin{defn}
    For cycles $[X_1], \,[X_2], \,\dots, \,[X_k] \in \mathbb{P}(V)$, we define their Apollonius cycles by
    \begin{equation}
        Ap([X_1],\,[X_2],\,\dots,\, [X_k]) = \langle [X_1],\,[X_2],\,\dots,\, [X_k]\rangle^{\perp} \cap Q,
    \end{equation}
    where $Q$ is the Lie quadric.
\end{defn}

\section{Main Results}\label{sec3}

In this section, we present the main results of the paper. We divide it into two parts. In Subsection~\ref{sec31}, we present theorems for handling the collinearity of the centers of solutions to the Apollonius problem. In Subsection~\ref{sec32}, we prove a theorem that generalizes Morita's theorem to $n$-dimensions and arbitrary circle configurations. 

\subsection{Concurrency of lines through the centers}\label{sec31}

First, we establish the first-level concurrency of lines through the centers of solutions to the Apollonius problem. In doing so, we introduce the point $P_X$, which plays a key role in all the results presented in this paper.

\begin{prop}\label{prop31}
    We are given $n+2$ cycles in $ \mathbb{R}^n$ in a set $X=\{X_1,\,\dots, \,X_{n+2}\}$. Let, for all $1 \leq i\leq n+2$ and cycles $X\backslash X_i$, there are exactly two distinct Apollonius cycles $A_i$ and $A_i'$ tangent to each of them. Then the lines through the centers of the pairs of cycles $A_1,\,A_1'$; $A_2,\,A_2';\,\dots$; and $A_{n+2}A_{n+2}'$ are concurrent.
\end{prop}

\begin{proof}
    We work in the framework of Lie sphere geometry. By abuse of notation, we assume that all cycles are represented as in Definition~\ref{representation}. Let $[P]=[p_1 : p_2 : \dots : p_{n+3}]\in \mathbb{P}(V)$ be a nonzero projective point such that
\begin{equation}
    [P] \in \langle [X_1],\,[X_2],\,\dots,\, [X_{n+2}]\rangle^{\perp}.
\end{equation}
Since this is a homogeneous system of $n+2$ linear equations in $n+3$ variables, such a $[P]$ exists.

Now define
\begin{equation}
     Y_1 = \langle [X_2],\,[X_3],\,\dots,\, [X_{n+2}]\rangle^{\perp}.
\end{equation}
Then $Y_1$ is a $2$-dimensional subspace of $\mathbb{P}(V)$. Since $[A_1]$ and $[A_1']$ are two distinct projective points of $Y_1$, we have
\begin{equation}
    Y_1=\langle [A_1],\,[A_1']\rangle.
\end{equation}
Therefore,
\begin{equation}\label{inspan}
    [P]\in \langle[A_1],\,[A_1']\rangle.
\end{equation}
Thus, projective points $[P]$, $[A_1]$, and $[A_1']$ are coplanar. Therefore, after projection, centers of $[P]$ and cycles $[A_1]$ and $[A_1']$ are collinear. Arguing analogously for each pair $[A_i]$ and $[A_i']$, we conclude that every line through the centers of $A_i$ and $A_i'$ passes through the center of $[P]$, which completes the proof.
\end{proof}

\begin{rem}
    We denote by $P_{X} = (p_2,\,\dots,\,p_{n+1},\,p_{n+2})$ the point of concurrency of these lines.
\end{rem}

Before proceeding, we state a lemma.

\begin{lem}\label{interchange}
        If $(P|P)\neq 0$, then $R_{P}([A_i])=[A_i']$, where $[A_i]$ and $[A_i']$ are defined as in Proposition~\ref{prop31}.
\end{lem} 
    \begin{proof}
    From Lemma~\ref{decomposition}, we know that there exist vectors $A_{i_\perp}$ and $A_{i_\perp}'$ such that
\begin{equation}
    \begin{split}
        (A_{i_\perp}| P) = 0, \qquad A_i &= A_{i_\perp} + \alpha_i P,\\
        (A_{i_\perp}'| P) = 0, \qquad A_i' &= A_{i_\perp}' + \alpha_i' P.
    \end{split}
\end{equation}
By~\eqref{inspan}, the projective points $[A_i]$, $[A_i']$, and $[P]$ are coplanar. If $[A_{i_\perp}] \neq [A_{i_\perp}']$, then the line through $A_i$ and $A_i'$ is generated by the span of $A_{i_\perp}$ and $A_{i_\perp}'$. Hence, there exist scalars $u$ and $v$ such that
\begin{equation}
    [P] = [uA_{i_\perp} + vA_{i_\perp}'].
\end{equation}
Therefore,
\begin{equation}
\begin{split}
    (P| P)
    &= (uA_{i_\perp} + vA_{i_\perp}' | P) \\
    &= u (A_{i_\perp}| P) + v (A_{i_\perp}'| P)\\
    &=0,
\end{split}
\end{equation}
a contradiction. Thus $[A_{i_\perp}] = [A_{i_\perp}']$. Since the Apollonius solutions are distinct, we have $\alpha_i \neq \alpha_i'$. Moreover, they satisfy
\begin{equation}\label{quadrat}
    \begin{split}
        (A_{i_\perp} + \alpha P | A_{i_\perp} + \alpha P) &= 0,\\
        (A_{i_\perp}| A_{i_\perp}) + 2\alpha (P| A_{i_\perp}) + \alpha^2 (P| P) &= 0,\\
        \alpha^2 (P| P) &= -(A_{i_\perp}| A_{i_\perp}).
    \end{split}
\end{equation}
Since $\alpha_i$ and $\alpha_i'$ are two distinct solutions of equation~\eqref{quadrat}, it follows that $\alpha_i' = -\alpha_i$. Consequently,
\begin{equation}
\begin{split}
    R_P([A_i'])
        &= R_P([A_{i_\perp} - \alpha_i P])\\
        &= \left[A_{i_\perp} - \alpha_i P - 2\frac{(A_{i_\perp} - \alpha_i P | P)}{(P| P)}P\right]\\
        &= \left[A_{i_\perp} - \alpha_i P - 2\frac{(A_{i_\perp}| P)}{(P| P)}P - 2\frac{-\alpha_i (P| P)}{(P| P)}P\right]\\
        &= \left[A_{i_\perp} + \alpha_i P\right]\\
        &= [A_i].
\end{split}
\end{equation}
Thus, $R_P([A_i']) = [A_i]$, and hence $R_P([A_i]) = [A_i']$.
    \end{proof}

\begin{thm}\label{main1}
    We are given $n+2$ cycles in $ \mathbb{R}^n$ in a set $X=\{X_1,\,\dots, \,X_{n+2}\}$. Let
    \begin{equation}
        \forall_{1\leq i\leq n+2}\colon Ap(X\setminus\{X_i\})=\{A_{i}, A_i'\}. 
    \end{equation}
    Let $A = \{A_1,\,A_2,\,\dots,\,A_{n+2}\}$ and $A' = \{A'_1,\, A'_2,\,\dots,\,A'_{n+2}\}$, assume that
    \begin{equation}
            \forall_{1\leq i\leq n+2}\colon  Ap(A\setminus \{A_i\}) = \{X_i, B_i\},\quad Ap(A'\setminus \{A_i'\}) = \{X_i, B_i'\}.
    \end{equation}
    Then the lines through the centers of the pairs of cycles $B_1$, $\,B_1';\, B_2,\,B_2';\,\dots;$\\$\,B_n,\,B_n'$, where $B_i \neq B_i'$ for all $i$, are concurrent at a point $P_X$.
\end{thm}
\begin{proof}
    Let $[P]$ be the projective point defined as in Proposition~\ref{prop31}. If $(P|P)=0$, then $[P]$ is a cycle tangent to all $[X_i]$. Thus, $n+2$ of $[A_1]$, $\dots,$ $[A_{n+2}],$ $[A_1'],$ $\dots,\,[A_{n+2}']$ are equal to $[P]$, so they are not linearly independent. Thus, for each $1 \leq i \leq n+2$, cycles $[B_i]$ and $[B'_i]$ are not clearly defined, a contradiction to the construction in the problem statement. Therefore, $(P|P)\neq 0$. Define $R_P$ as in Definition~\ref{def27}. First observe that
\begin{equation}
    \forall_{1\leq i\leq n+2}\colon R_P([X_i])=\left[ X_i-2\frac{(X_i|P)}{(P|P)}P\right]=[X_i].
\end{equation}
Thus, the set of cycles $X$ is fixed under this transformation. We now show that $R_P([B_i])=[B_i']$. Since $[B_i]\neq [B_i']$, either $[B_i]\neq [X_i]$ or $[B_i']\neq [X_i]$. Without loss of generality, assume that $[B_i]\neq [X_i]$. From Lemma~\ref{properties} and Lemma~\ref{interchange} we know that
\begin{equation}
    \begin{split}
        0&=(B_i|A_j)=(R_P(B_i)|R_P(A_j))=(R_P(B_i)|A_j'),\\
        0&=(B_i|B_i)=(R_P(B_i)|R_P(B_i)).
    \end{split}
\end{equation}
Thus, $R_P([B_i])$ is tangent to all cycles in $A'\setminus\{A_i'\}$ and lies on the quadric $Q$. Hence
\begin{equation}
    R_P([B_i])\in Ap(A'\setminus\{A_i'\})=\{[X_i],\,[B_i']\}.
\end{equation}
Moreover, $R_P([B_i])\neq [X_i]$, since $[X_i]$ is fixed by $R_P$, and otherwise we would obtain $[B_i]=[X_i]$, a contradiction. Therefore, $R_P([B_i])=[B_i']$.

By Lemma~\ref{decomposition}, we have
\begin{equation}
    [B_i]=[P_\perp+\alpha P].
\end{equation}
Therefore,
\begin{equation}
    \begin{split}
        [B_i']&=R_P([B_i])\\
            &=R_P([P_\perp+\alpha P])\\
            &=\left[P_\perp+\alpha P-2\frac{(P_\perp|P)}{(P|P)}P-2\frac{\alpha(P|P)}{(P|P)}P\right]\\
            &=[P_\perp-\alpha P].
    \end{split}
\end{equation}
Thus, $[P]$, $[B_i]$, and $[B_i']$ all lie in $\langle P,P_\perp\rangle$. Since $[B_i]\neq [B_i']$, it follows that
\begin{equation}
    [P]\in \langle [B_i],\,[B_i']\rangle.
\end{equation}
As in the proof of Proposition~\ref{prop31}, we conclude that $P_X$ lies on the line through the centers of the pair of cycles $B_i$ and $B_i'$.
\end{proof}
\begin{rem}
      Observe that in the above theorem the labels $A_i$ and $A_i'$ may be interchanged arbitrarily. Thus, for each set $A\backslash{A_i}$, there are $\frac{2^{n+1}}{2}=2^n$ possible to assign primes. We divide by $2$ because simultaneously reversing the primes in all pairs $(A_i, A_i')$ yields the same line through the centers, since it merely interchanges the cycles $B_i$ and $B_i'$. Therefore, there are $n\cdot 2^n$ lines constructed as in Theorem~\ref{main1}, all passing through the point $P_X$.
\end{rem}

\subsection{Existence of cycle tangent to given hyperplanes}\label{sec32}

This subsection is devoted to the second main result of the paper. We prove the existence of a cycle tangent to given sets of hyperplanes. For instance, in two dimensions these are sets of two lines, while in three dimensions they are sets of cones, viewed as sets of hyperplanes tangent to them. In the two and three dimensions, a degeneration of this theorem yields K. Morita's theorem, which we present in Corollary~\ref{mort}~\cite{general}. 

    \begin{thm}\label{main2}
        We are given $n+2$ cycles in $ \mathbb{R}^n$ in a set $X=\{X_1, \,\dots, \,X_{n+2}\}$. Let
    \begin{equation}
        \forall_{1\leq i\leq n+2}\colon Ap(X\setminus\{X_i\})=\{A_{i}, A_i'\}. 
    \end{equation}
    We define  $T_i$ $(1 \leq i \leq n+2)$ as the set of hyperplanes tangent to two cycles $A_i$ and $A_i'$. We assume that the point $P_X$, the point of concurrency of the lines through the centers of the solutions to the Apollonius problem, is not a point at infinity. Then there exists a sphere with the center $P_X$ tangent to all hyperplanes in sets $T_1,\,T_2,\,\dots,\,T_{n+2}$.
    \end{thm}
    \begin{proof}
            We again set
        \begin{equation}
        [P] = [p_1 : p_2 : \dots : p_{n+3}],
        \end{equation}
        for the projective point $[P]$ from Proposition~\ref{prop31}. Since $P_X$ is not a point at infinity, we have $p_{n+2}\neq0$. Define
        \begin{equation}
        [P'] = \left[p_1 : p_2 : \dots : p_{n+2} : p_{n+3}-\frac{(P|P)}{2p_{n+2}}\right].
        \end{equation}
        Then
        \begin{equation}
        \begin{split}
        (P'|P') &= \left(P-\left(0,\,\dots,\,0,\,\frac{(P|P)}{2p_{n+2}}\right)\middle|P-\left(0,\,\dots,\,0,\,\frac{(P|P)}{2p_{n+2}}\right)\right)\\
        &= (P|P) - 2\left(P\middle|\left(0,\,\dots,\,0,\,\frac{(P|P)}{2p_{n+2}}\right)\right)\\
        &= (P|P) - (P|P)\\
        &= 0.
        \end{split}
        \end{equation}
        Thus, $[P']$ is a cycle. Since $p_{n+2}\neq0$, $[P']$ is a sphere. Moreover, by the construction of $P'$, the point $P_X$ is the center of the cycle $P'$.
        Let $[T] = [t_1 : t_2 : \dots : t_{n+1} : 0 : t_{n+3}]$ be a hyperplane tangent to $[A_1]$ and $[A_1']$. From~\eqref{inspan} we know that
        \begin{equation}
        [P] \in \langle [A_1],\,[A_1']\rangle.
        \end{equation}
        Hence
        \begin{equation}
        \begin{split}
        (P'|T) &= \left(uA_1+vA_1'-\left(0,\,\dots,\,0,\,\frac{(P|P)}{2p_{n+2}}\right)\middle|T\right)\\
        &= u(A_1|T)+v(A_1'|T)-\left(\left(0,\,\dots,\,0,\,\frac{(P|P)}{2p_{n+2}}\right)\middle|T\right)\\
        &= 0.
        \end{split}
        \end{equation}
        Analogously, we prove that $[P']$ is tangent to each hyperplane from the sets $T_i$. Thus, $[P']$ is the sphere from the statement.

\end{proof}

\begin{rem}
    We prove the above theorems for spheres in $\mathbb{R}^n$. By applying an affine shear transformation, we can replace all spheres, together with their Apollonius solutions, by similar ellipsoids with the same orientation, so that spheres appear as a special case.
\end{rem}

\section{Degenerations and Aplications}\label{sec4}
In this section, we present applications of the theorems established in the previous section. For the sake of clarity, both the results and the accompanying figures are presented in the plane. It should be emphasized, however, that all of the results obtained here extend directly to dimension $n$, yielding natural counterparts of these theorems in every dimension. This section is divided into two main parts. The first, presented in Subsection~\ref{subsec42}, deals with applications to classical triangle geometry. The second, presented in Subsection~\ref{subsec41}, concerns the classical configuration of mutually tangent circles, which has been the subject of numerous studies~\cite{tobefilled,general,kissprecise,ijcdm}. 

\subsection{Degenerations in a triangle}\label{subsec42}

We now turn to applications of the proved theorems in triangle geometry. We begin with very elementary facts, namely the existence of the incenter and the circumcenter, and then proceed to several further, more intricate theorems. These results admit direct analogues in higher dimensions; for example, in three dimensions, a triangle is replaced by a tetrahedron.

We also note that every triangle may be replaced by a circular horn, that is, a triangle whose sides are circular arcs~\cite{horn}. This does not affect the proofs, since our theorems concern cycles.

\begin{figure}[h]
    \centering
  \begin{minipage}{0.45\textwidth}
    \centering
    \input{example1fig}
    \caption{Example~\ref{check1}: Circumcenter.}
    \label{fig11}
  \end{minipage}\hfill
  \begin{minipage}{0.45\textwidth}
    \centering
    \input{example2fig}
    \caption{Example~\ref{check2}: Incenter.}
    \label{fig0}
  \end{minipage}
\end{figure}

\begin{ex}[Circumcenter]\label{check1}
    The perpendicular bisectors of a triangle are concurrent at a single point, the circumcenter.
\end{ex}
\begin{proof}
    Let $ABC$ be a triangle and let $\omega$ be a circle (Fig.~\ref{fig11}). By applying Proposition~\ref{prop31} to
    \begin{equation}
    X = \{A,\,B,\,C,\,\omega\},
    \end{equation}
    we obtain that the corresponding lines determined by the solutions to the Apollonius problem are concurrent. These lines are the perpendicular bisectors of $BC$, $CA$, and $AB$, which completes the proof.
\end{proof} 

\begin{ex}[Incenter]\label{check2}
    The angle bisectors of a triangle are concurrent at a single point, the incenter.
\end{ex}

\begin{proof}
    Let $ABC$ be a triangle, and let $\omega$ be a circle inside it. We orient cycles as in Fig.~\ref{fig0}. By applying Proposition~\ref{prop31} to
    \begin{equation}
         X = \{BA,\,AC,\,CB,\,\omega\},
    \end{equation}
    we obtain that the corresponding lines determined by the solutions to the Apollonius problem are concurrent at a point $I$. These lines are precisely the bisectors of the angles $BAC$, $ACB$, and $CBA$, which completes the proof.
\end{proof}

\begin{rem}
    By applying Theorem~\ref{main1} in the above proof, we obtain that $I$ is the center of the inscribed circle of triangle $ABC$.
\end{rem} 

We now present the theorem that motivated this paper. As far as we are aware, this result is new. The case when the circle $P$ in the following corollary degenerates to a point was added by the author to the Encyclopedia of Triangle Centers  $($$X(1)$~\cite{kimberling}$)$. The result follows immediately from the proved theorems.

\begin{cor}[Outer Apollonius' Circles Theorem in a Triangle]\label{outerapol}
Let $ABC$ be a triangle and let $\omega$ be any circle with center $P$ inside triangle $ABC$. Let $\omega_a$ be the smaller of the two circles tangent to $\omega$ and to the lines $AB$ and $AC$; define $\omega_b$ and $\omega_c$ cyclically. Similarly, let $\omega'_a$ be the larger of the two circles tangent to $\omega$ and to the lines $AB$ and $AC$; define $\omega'_b$ and $\omega'_c$ cyclically. Let $O_1$ be the outer Apollonius' circle of the $\omega_a$, $\omega_b$, and $\omega_c$, and let $O_2$ be the outer Apollonius' circle of the $\omega'_a$, $\omega'_b$, and $\omega'_c$. Then the line $O_1O_2$ passes through the fixed point $I$, the incenter of $\triangle ABC$.
\end{cor}

\begin{figure}[h]
    \centering
  \begin{minipage}{0.45\textwidth}
    \centering
   \includegraphics[scale=0.55]{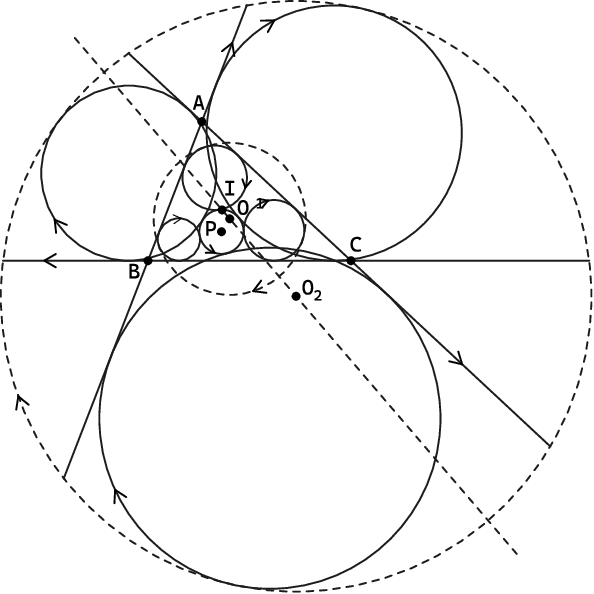}
    \caption{Corollary~\ref{outerapol}: Points $I$, $O_1$, and $O_2$ are collinear.}
     \label{cor49fig}
  \end{minipage}\hfill
  \begin{minipage}{0.45\textwidth}
    \centering
    \input{corollary49fig2}
    \caption{Corollary~\ref{pozdro}: Points $I$, $O$, and $S$ are collinear.}
    \label{mix}
  \end{minipage}
\end{figure}

\begin{proof}
        We orient the cycles as in Fig.~\ref{cor49fig}. Observe that
    \begin{equation}
        \begin{split}
            Ap(\omega,\,BA,\,AC) &= (\omega_a,\, \omega_a'),\\
            Ap(\omega,\,CB,\,BA) &= (\omega_b,\, \omega_b'),\\
            Ap(\omega,\,AC,\,CB) &= (\omega_c,\, \omega_c').
        \end{split}
    \end{equation}
    Applying Theorem~\ref{main1} to
    \begin{equation}
        X = \{\omega,\, BA,\, AC,\, CB\},
    \end{equation}
    and sets
    \begin{equation}
        \begin{split}
            B_1  &= Ap(\omega_a,\,\omega_b,\,\omega_c)\setminus\{\omega\},\\
            B_1' &= Ap(\omega_a',\,\omega_b',\,\omega_c')\setminus\{\omega\},
        \end{split}
    \end{equation}
    we obtain that the line $O_1O_2$ passes through the point $I$, the point $P_X$ from Theorem~\ref{main1}.
\end{proof}

This theorem was of particular interest to us because of the following corollary.

\begin{cor}[A Duality for the Collinearity of Centers]
    The points $P$, $I$, and $O_1$ are collinear if and only if the points $P$, $I$, and $O_2$ are collinear.
\end{cor}

This result allows one to pass from one collinearity relation to another, which is occasionally useful when reducing a difficult problem in this configuration to a simpler one. For instance, one may take $P$ to be the Gergonne point of triangle $ABC$, for which one of the two collinearities is easy to prove by homothety, and then deduce the other automatically (see X(7), X(30332))~\cite{kimberling,montesdeoca}.

As a final example in this subsection, we present another interesting degeneration leading to a collinearity result in a configuration involving mixtilinear circles\footnote{A mixtilinear circle is a circle tangent to two sides of a triangle and to its circumcircle.} (Fig.~\ref{mix}). The corresponding generalization to higher dimensions is also of interest, especially for mixtilinear spheres in three dimensions, where such a result appears to be new.

\begin{cor}\label{pozdro}
    There is given $\triangle ABC$ with a circumcircle $\Omega$ with center $O$ and incenter $I$. Circle $\omega$ with center $S$ is tangent internally to the circles $\omega_A$, $\omega_B$, and $\omega_C$, which are three mixtilinear circles in $\triangle ABC$. Then points $S$, $I$ and $O$ are collinear.
\end{cor} 

\begin{proof}
   We orient the cycles as in Fig.~\ref{mix}. Observe that
\begin{equation}
    \begin{split}
        Ap(\Omega,\,BA,\,AC) &= (A,\, \omega_A),\\
        Ap(\Omega,\,CB,\,BA) &= (B,\, \omega_B),\\
        Ap(\Omega,\,AC,\,CB) &= (C,\, \omega_C).
    \end{split}
\end{equation}
Moreover, the lines through the centers of the pairs $(A,\,\omega_A)$, $(B,\,\omega_B)$, and $(C,\,\omega_C)$ intersect at the point $I$.
Applying Theorem~\ref{main1} to
\begin{equation}
    X = \{\Omega,\, CB,\, BA,\, AC\},
\end{equation}
and to the sets
\begin{equation}
    \begin{split}
        B_1  &= Ap(A,\,B,\,C)\setminus\{\Omega\} = \{-\Omega\},\\
        B_1' &= Ap(\omega_A,\,\omega_B,\,\omega_C)\setminus\{\Omega\} = \{\omega\},
    \end{split}
\end{equation}
we obtain that the line $SO$ passes through the point $I$.
\end{proof}

\subsection{Kissing circles}\label{subsec41}
In this subsection, we derive a few direct consequences of our theorem for the well-studied configuration of kissing circles, that is, circles externally tangent to one another and thus appearing to "kiss"\footnote{The term derives from Frederick Soddy's note in Nature, \textit{The Kiss Precise}~\cite{kissprecise}.}. To simplify the exposition, we first introduce the notion of the Soddy line (see the line $O_1O_2$ in Fig.~\ref{moritafig})~\cite{sodyline}.

\begin{defn}[Soddy line]
    The Soddy line is the line through the centers of the inner and outer Apollonius circles in a configuration of three mutually tangent circles.
\end{defn}

Many papers have been devoted to points lying on this line~\cite{ijcdm,sodyline}. As we shall see, our theorems provide a natural framework for understanding why certain points lie on this line. This also leads naturally to a generalization of the Soddy line for configurations of spheres in $\mathbb{R}^n$.

\begin{defn}[Soddy line in $\mathbb{R}^n$]\label{gensody}
    Let $S_1, \,S_2,\, \dots,\, S_{n+1}$ be $n+1$ spheres externally tangent to each other in $\mathbb{R}^n$. Then their Soddy line is the line through the centers of their inner and outer Apollonius spheres.
\end{defn}

Let us see that Proposition~\ref{prop31} and Theorem~\ref{main1} provides us with a very useful tool to study points lying on the Soddy line, as a Soddy line is just a special case of line thourogh centers of two Apollonius circles for mutually tangent circles. We will see direct results

Proposition~\ref{prop31} and Theorem~\ref{main1} provide a very useful tool for studying points lying on the Soddy line, since the Soddy line in $\mathbb{R}^n$ is a special case of the line through the centers of two Apollonius spheres associated with mutually tangent spheres. We now present some direct consequences.

First, we present K. Morita's theorem on kissing circles, proved in dimensions two and three in~\cite{general}. Fig.~\ref{moritafig} illustrates its two-dimensional version for kissing circles, while Fig.~\ref{gencon} shows the corresponding two-dimensional version for a general configuration, obtained directly from Theorem~\ref{main2}.

\begin{figure}[h]
    \centering
  \begin{minipage}{0.45\textwidth}
    \centering
     \input{corollary43}
    \caption{Corollary~\ref{mort}: K. Morita's Theorem in 2-dimensions.}
    \label{moritafig}
  \end{minipage}\hfill
  \begin{minipage}{0.45\textwidth}
    \centering
    \input{corollary43fig2}
    \caption{Corollary~\ref{mort}: K. Morita's Theorem in general configuration.}
    \label{gencon}
  \end{minipage}
\end{figure}

\begin{cor}[K. Morita's Theorem]\label{mort}
    Let $S_{O1}$, $S_{O2}$, $S_{O3}$, and $S_{O4}$ be four spheres externally tangent to each other and $S$ be an outer sphere tangent to $S_{O1}$, $S_{O2}$, $S_{O3}$, and $S_{O4}$ $($see the figure in~\cite{general}$)$. Construct new spheres $S_I1$ tangent to $S_{O2}$, $S_{O3}$, $S_{O4}$, and $S$, $S_{I2}$ tangent to $S_{O3}$, $S_{O4}$, $S_{O1}$, and $S$, $S_{I3}$ tangent to $S_{O4}$, $S_{O_1}$, $S_{O2}$, and $S$, and $S_{I4}$ tangent to $S_{O1}$, $S_{O2}$, $S_{O3}$, and $S$. Denote by $\{S_{O1}S_{I1}\}$ the common external tangent cone of opposite spheres $S_{O1}$ and $S_{I1}$, and define $\{S_{O2}S_{I2}\}$, $\{S_{O3}S_{I3}\}$, and $\{S_{O4}S_{I4}\}$ similarly. Then there exists a unique sphere inscribed in $\{S_{O1}S_{I1}\}$, $\{S_{O2}S_{I2}\}$, $\{S_{O3}S_{I3}\}$, and $\{S_{O4}S_{I4}\}$.
\end{cor}

\begin{proof}
    We work in three dimensions. We orient spheres $S_{O1},\, S_{O2},\, S_{O3},$ and $S_{O4}$ in the positive direction, and we orient spheres $S,\,S_{I1},\, S_{I2},\, S_{I3},\, S_{I4}$ in the negative direction. Observe that
    \begin{equation}
        \begin{split}
            Ap(S_{O1},\,S_{O2},\, S_{O3},\, S) &= \{S_{I4},\,-S_{O4}\},\\
            Ap(S_{O1},\,S_{O2},\, S_{O4},\, S) &= \{S_{I3},\,-S_{O3}\},\\
            Ap(S_{O1},\,S_{O3},\, S_{O4},\, S) &= \{S_{I2},\,-S_{O2}\},\\
            Ap(S_{O2},\,S_{O3},\, S_{O4},\, S) &= \{S_{I1},\,-S_{O1}\}.
        \end{split}
    \end{equation}
    Clearly, in labels from Theorem~\ref{main2}, we have that 
\begin{equation}
        \begin{split}
            T_1 = \{S_{O1}S_{I1}\},\\
            T_2 = \{S_{O2}S_{I2}\},\\
            T_3 = \{S_{O3}S_{I3}\},\\
            T_4 = \{S_{O4}S_{I4}\}.\\
        \end{split}
    \end{equation}
    Then the result follows by applying Theorem~\ref{main2} to the set 
    \begin{equation}
        X=\{S_{O1},\,S_{O2},\,S_{O3},\,S_{O4},\,S\}.
    \end{equation}
\end{proof}
\begin{rem}
    Clearly, Corollary~\ref{mort} is a special case of Theorem~\ref{main2}. The latter yields an $n$-dimensional generalization of Corollary~\ref{mort}. Moreover, it gives an additional property: the center of the inscribed sphere lies on the Soddy line determined by the given $n+2$ spheres; in the case of Corollary~\ref{mort}, this is the Soddy line of $S_{O1}$, $S_{O2}$, $S_{O3}$, and $S_{O4}$.
\end{rem}

We now show how our results can be used to study geometric properties of Apollonian gaskets~\cite{gasket}. As an illustration, we apply them to prove Corollary~\ref{cor45soddy}, which was presented in~\cite{ijcdm}. This directly generalizes that result to $n$-dimensional Apollonian gaskets.

\begin{figure}[h]
    \centering
    \input{col46fig}
    \label{fig1}
    \caption{Corollary~\ref{cor45soddy}: Concurrent lines on the Soddy line}
\end{figure}
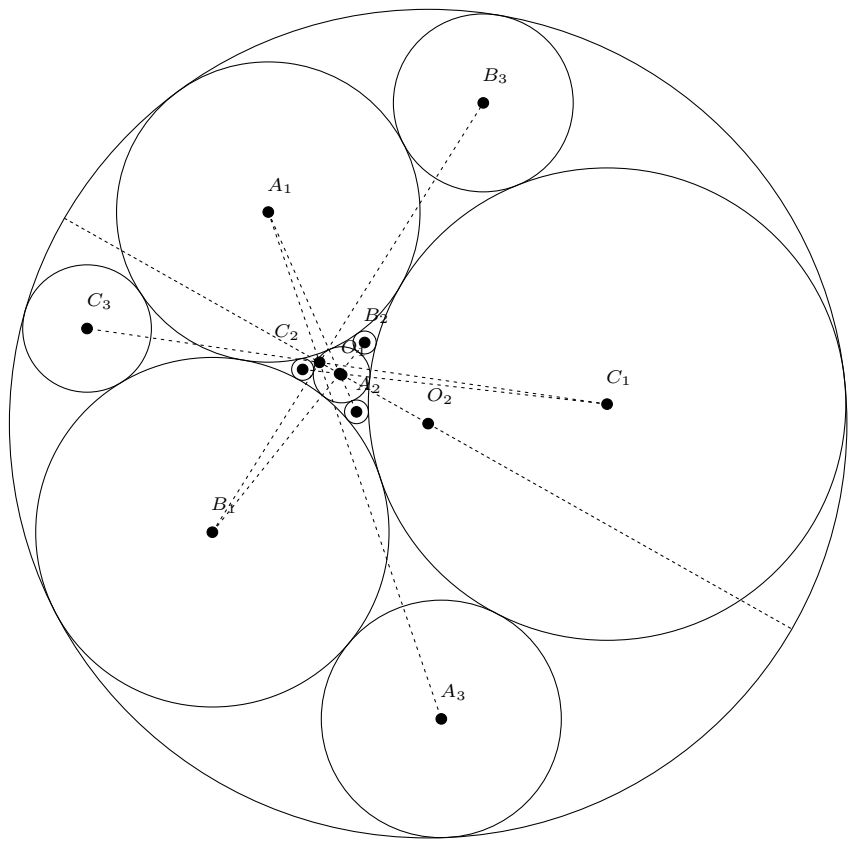
 
\begin{cor}[A Property of the Apollonian Gasket]\label{cor45soddy}
    Let $\Omega_1$, $\Omega_2$, and $\Omega_3$ be three circles with centers $A_1$, $B_1$, and $C_1$, respectively, that are externally tangent to one another. Let $S_1$ and $S_2$ be the inner and outer Apollonius circles, with centers $O_1$ and $O_2$, respectively. Let the circles $\omega_1$, $\omega_2$, $\omega_3$, centered at $A_2$, $B_2$, and $C_2$, distinct from $\Omega_1$, $\Omega_2$, and $\Omega_3$, respectively, be tangent to the triples of circles $(\Omega_2,\,\Omega_3,\,S_1)$, $(\Omega_3,\,\Omega_1,\,S_1)$, and $(\Omega_1,\,\Omega_2,\,S_1)$, respectively. Define the circles $\omega_1'$, $\omega_2'$, $\omega_3'$, centered at $A_3$, $B_3$, and $C_3$ analogously, with $S_2$ in place of $S_1$. Then the lines $A_1A_2$, $B_1B_2$, and $C_1C_2$ are concurrent, and the lines $A_1A_3$, $B_1B_3$, and $C_1C_3$ are concurrent. Moreover, these two points of concurrency lie on the line $O_1O_2$, the Soddy line of the circles $\Omega_1$, $\Omega_2$, and $\Omega_3$.
\end{cor}

\begin{proof}
    Let the circles $\Omega_1,\, \Omega_2,\, \Omega_3,\, \omega_1,\, \omega_2,\, \omega_3,\,S_1$ be oriented in the same direction. Observe that
\begin{equation}
    \begin{split}
        Ap(\Omega_2,\,\Omega_3,\,S_1)&=\{-\Omega_1,\,-\omega_1\},\\
        Ap(\Omega_3,\,\Omega_1,\,S_1)&=\{-\Omega_2,\,-\omega_2\},\\
        Ap(\Omega_1,\,\Omega_2,\,S_1)&=\{-\Omega_3,\,-\omega_3\}.
    \end{split}
\end{equation}
Applying Proposition~\ref{prop31} to the set
\begin{equation}
    X=\{\Omega_1,\, \Omega_2,\, \Omega_3,\, S_1\},
\end{equation}
 we obtain that the lines $A_1A_2$, $B_1B_2$, $C_1C_2$, and $O_1O_2$ are concurrent. Analogously, we obtain the concurrency of the lines $A_1A_3$, $B_1B_3$, $C_1C_3$, and $O_1O_2$, which completes the proof.
\end{proof}

\begin{rem}
    By applying Theorem~\ref{main1}, we obtain further properties. For instance, in the configuration of the above corollary, the center of a circle internally tangent to the circles centered at $A_2$, $B_2$, and $C_2$ lies on the line $O_1O_2$. Likewise, the center of a circle externally tangent to the circles centered at $A_3$, $B_3$, and $C_3$ lies on the same line.
\end{rem}

    As the last corollary in this section, we leave the reader with a problem from the olympiad geometry book~\cite{Akopyan2017}. It is an immediate consequence of our result.
    
\begin{cor}
   Let $\omega_1$ and $\omega_2$ be externally tangent circles with centers $O_1$ and $O_2$, both internally tangent to the circle $\Omega$. Let $\omega_3$ and $\omega_4$ be circles with centers $O_3$ and $O_4$, externally tangent to both $\omega_1$ and $\omega_2$ and internally tangent to $\Omega$. Let $\omega_5$ and $\omega_6$ be circles with centers $O_5$ and $O_6$, externally tangent to $\omega_3$ and $\omega_1$, and to $\omega_3$ and $\omega_2$, respectively, and internally tangent to $\Omega$. Then the lines $O_1O_6$, $O_2O_5$, and $O_3O_4$ are concurrent.
\end{cor}

% ------------------------------------------------------------------------
\end{document}

%% file: example1fig.tex
\begin{tikzpicture}[scale = 0.4, line cap=round,line join=round,>=triangle 45,x=1cm,y=1cm]
\clip(-10.96660714258106,-6.5046031532462335) rectangle (2.4200378170656247,7.765828049576537);
\draw [line width=0.4pt] (-2.349211096892333,1.7622845523406627) circle (2cm);
\draw [line width=0.4pt] (-2.1583953691302273,-5.781514222855904) circle (5.546211672264367cm);
\draw [line width=0.4pt] (-2.2149196377744165,-0.4124009980761101) circle (4.178827996695348cm);
\draw [line width=0.4pt] (0.7166391085960795,3.7422228482941877) circle (5.6496017780405685cm);
\draw [line width=0.4pt] (-5.643453491159806,-2.512443716866337) circle (7.3967892980685725cm);
\draw [line width=0.4pt] (-3.997069049512911,1.186021920119912) circle (3.7457131646719386cm);
\draw [line width=0.4pt] (-10.406109603710767,2.417158997189638) circle (6.083469186413062cm);
\draw [line width=0.4pt] (-4.825483767343133,4.838979101264928)-- (-6.12,-1.9);
\draw [line width=0.4pt] (-6.12,-1.9)-- (1.7206136462119783,-1.8174565829552756);
\draw [line width=0.4pt] (1.7206136462119783,-1.8174565829552756)-- (-4.825483767343133,4.838979101264928);
\draw [line width=0.4pt,dash pattern=on 1pt off 2pt] (-10.406109603710767,2.417158997189638)-- (-2.228170065615906,0.846227358632837);
\draw [line width=0.4pt,dash pattern=on 1pt off 2pt] (-2.228170065615906,0.846227358632837)-- (-2.1583953691302273,-5.781514222855904);
\draw [line width=0.4pt,dash pattern=on 1pt off 2pt] (-5.643453491159806,-2.512443716866337)-- (0.7166391085960795,3.7422228482941877);
\begin{scriptsize}
\draw [fill=black] (-6.12,-1.9) circle (2.5pt);
\draw[color=black] (-6.448614468700304,-1.3607108456714399) node {$B$};
\draw [fill=black] (1.7206136462119783,-1.8174565829552756) circle (2pt);
\draw[color=black] (1.8511054062806407,-1.4611106828687899) node {$C$};
\draw [fill=black] (-4.825483767343133,4.838979101264928) circle (2pt);
\draw[color=black] (-4.741817236345352,5.349344940351461) node {$A$};
\draw [fill=black] (-10.406109603710767,2.417158997189638) circle (2pt);
\draw [fill=black] (-2.1583953691302273,-5.781514222855904) circle (2pt);
\draw [fill=black] (0.7166391085960795,3.7422228482941877) circle (2pt);
\draw [fill=black] (-3.997069049512911,1.186021920119912) circle (2pt);
\draw [fill=black] (-5.643453491159806,-2.512443716866337) circle (2pt);
\draw [fill=black] (-2.2149196377744165,-0.4124009980761101) circle (2pt);
\draw [fill=black] (-2.228170065615906,0.846227358632837) circle (2pt);
\draw[color=black] (-2.4158876746067404,1.350084758657014) node {$O$};
\draw[color=black] (-3.4158876746067404,2.650084758657014) node {$\omega$};
\end{scriptsize}
\end{tikzpicture}

%% file: example2fig.tex
\definecolor{wewdxt}{rgb}{0.43137254901960786,0.42745098039215684,0.45098039215686275}
\definecolor{rcrcrf}{rgb}{0.10980392156862745,0.10980392156862745,0.12156862745098039}
\begin{tikzpicture}[scale=0.25, line cap=round,line join=round,>=triangle 45,x=1cm,y=1cm]
\clip(-52.23215160353097,-19.35674923956194) rectangle (-30.24107601972917,4.23248922953594);
\draw [line width=0.4pt,color=rcrcrf] (-42.9284964307375,-6.265682800797273) circle (1cm);
\draw [line width=0.4pt,color=rcrcrf] (-40.91467636666798,-14.144469846619952) circle (7.13208193292901cm);
\draw [line width=0.4pt,color=rcrcrf] (-43.24502347166313,-3.9925802140379574) circle (1.2950348009493495cm);
\draw [line width=0.4pt,color=rcrcrf] (-40.564458827448,-6.622106582571969) circle (1.3907554667889592cm);
\draw [line width=0.4pt,color=rcrcrf] (-47.45618698977164,-3.908175023204252) circle (4.104686544713035cm);
\draw [line width=0.4pt,color=rcrcrf] (-44.75662169766572,-7.0310048035912684) circle (0.9818576537028604cm);
\draw [line width=0.4pt,color=rcrcrf] (-37.33514334772426,-1.9451421762609726) circle (6.067720283056533cm);
\draw [line width=0.4pt,color=rcrcrf] (-52.77711649361783,-8.012862458452464)-- (-26.82767272461651,-8.012862458452464);
\draw [line width=0.4pt,color=rcrcrf] (-47.68137907026714,1.9131449008150199)-- (-28.959209023380556,-15.538592178604109);
\draw [line width=0.4pt,color=rcrcrf] (-41.35308489797593,4.484903598377278)-- (-49.62437428963426,-16.88941849293763);
\draw [line width=0.4pt] (-50.86945743731461,-6.683511508918704)-- (-50.934810160995305,-6.087081277032999);
\draw [line width=0.4pt] (-50.934810160995305,-6.087081277032999)-- (-50.38561006675689,-6.32869927406186);
\draw [line width=0.4pt] (-46.883079495416624,-17.463735710108832)-- (-47.38763122427309,-17.139039616957387);
\draw [line width=0.4pt] (-47.38763122427309,-17.139039616957387)-- (-47.41655042502357,-17.73834227824617);
\draw [line width=0.4pt] (-40.787413640788074,3.40120889546661)-- (-40.18748340491276,3.4103583261732835);
\draw [line width=0.4pt] (-40.18748340491276,3.4103583261732835)-- (-40.479524883428276,2.886228786053536);
\draw [line width=0.4pt] (-50.42922915481427,-7.712862458452394)-- (-50.948844397085054,-8.012862458452464);
\draw [line width=0.4pt] (-50.948844397085054,-8.012862458452464)-- (-50.42922915481427,-8.31286245845253);
\draw [line width=0.4pt] (-42.47983882912396,2.404462290038277)-- (-42.01253081127489,2.7807908917334396);
\draw [line width=0.4pt] (-42.01253081127489,2.7807908917334396)-- (-41.92027469096074,2.187925976036413);
\draw [line width=0.4pt] (-33.04784872570329,-11.317273796061304)-- (-32.8723114465835,-11.891021705547107);
\draw [line width=0.4pt] (-32.8723114465835,-11.891021705547107)-- (-33.45696035112632,-11.756167493833143);
\draw [line width=0.4pt] (-42.155913791129954,-4.016910735202637)-- (-42.01258973469601,-4.390351781215864);
\draw [line width=0.4pt] (-42.01258973469601,-4.390351781215864)-- (-41.760842330249694,-4.079508984364019);
\draw [line width=0.4pt] (-43.441763852838086,-6.870723402707389)-- (-43.23980113532254,-7.2159929506447975);
\draw [line width=0.4pt] (-43.23980113532254,-7.2159929506447975)-- (-43.63979469374728,-7.2182630206619);
\draw [line width=0.4pt] (-45.01208507390414,-5.88875799377018)-- (-44.64832977952581,-6.055137336872317);
\draw [line width=0.4pt] (-44.64832977952581,-6.055137336872317)-- (-44.97429616450639,-6.286968991013967);
\draw [line width=0.4pt] (-41.23198480165472,-5.186639535359252)-- (-40.83851468429872,-5.258620551939755);
\draw [line width=0.4pt] (-40.83851468429872,-5.258620551939755)-- (-41.097587131925664,-5.563385160909842);
\draw [line width=0.4pt,dash pattern=on 1pt off 2pt] (-43.762045976835054,-1.740233446991293)-- (-40.91467636666798,-14.144469846619952);
\draw [line width=0.4pt,dash pattern=on 1pt off 2pt] (-46.18938517284112,-8.012862458452458)-- (-37.33514334772426,-1.9451421762609726);
\draw [line width=0.4pt,dash pattern=on 1pt off 2pt] (-47.45618698977164,-3.908175023204252)-- (-37.032788799788534,-8.012862458452458);
\begin{scriptsize}
\draw [fill=rcrcrf] (-43.762045976835054,-1.740233446991293) circle (2pt);
\draw[color=rcrcrf] (-43.37733798350458,-0.7073163849855991) node {$A$};
\draw [fill=rcrcrf] (-46.18938517284112,-8.012862458452458) circle (2pt);
\draw[color=rcrcrf] (-45.809979087907436,-6.98353043434568) node {$B$};
\draw [fill=rcrcrf] (-37.032788799788534,-8.012862458452458) circle (2pt);
\draw[color=rcrcrf] (-36.66324853535271,-6.98353043434568) node {$C$};
\draw [fill=black] (-40.91467636666798,-14.144469846619952) circle (2pt);
\draw [fill=black] (-47.45618698977164,-3.908175023204252) circle (2pt);
\draw [fill=black] (-37.33514334772426,-1.9451421762609726) circle (2pt);
\draw [fill=black] (-42.84782609948943,-5.722926610303264) circle (2pt);
\draw[color=black] (-42.501587185919554,-4.794153440382861) node {$I$};
\draw [fill=black] (-44.75662169766572,-7.0310048035912684) circle (2pt);
\draw [fill=black] (-43.24502347166313,-3.9925802140379574) circle (2pt);
\draw [fill=black] (-40.564458827448,-6.622106582571969) circle (2pt);
\end{scriptsize}
\end{tikzpicture}

%% file: corollary49fig2.tex
\definecolor{wewdxt}{rgb}{0.43137254901960786,0.42745098039215684,0.45098039215686275}
\begin{tikzpicture}[scale=0.5, transform shape,line cap=round,line join=round,>=triangle 45,x=1cm,y=1cm]
\clip(-9.091803373987176,-6.276577226578134) rectangle (2.7885129013857997,5.797741103058356);
\draw [line width=0.4pt,color=wewdxt] (-3.165727382224656,-0.25932447791768926) circle (5.559492893345928cm);
\draw [line width=0.4pt,color=wewdxt] (-4.616925493038765,-2.2951709952333803) circle (3.0593634969152443cm);
\draw [line width=0.4pt,color=wewdxt] (-2.9430848809800483,0.947823769962258) circle (4.331984614058134cm);
\draw [line width=0.4pt,color=wewdxt] (-5.995915089567697,-0.6256133924734061) circle (2.705700674609243cm);
\draw [line width=0.4pt,color=wewdxt] (-4.940475892414381,-0.890186688699903) circle (1.6176056245450336cm);
\draw [line width=0.4pt,color=wewdxt] (-6.961551841554828,3.8026550770671337)-- (-7.82,-3.3);
\draw [line width=0.4pt,color=wewdxt] (-7.82,-3.3)-- (1.38,-3.46);
\draw [line width=0.4pt,color=wewdxt] (1.38,-3.46)-- (-6.961551841554828,3.8026550770671337);
\draw [line width=0.4pt,dash pattern=on 1pt off 2pt] (-8.404111842746705,-2.1213903937555236)-- (2.072657078297394,1.6027414379201455);
\draw [line width=0.4pt] (-6.224305741171919,4.560671170910972)-- (-5.924382412942908,4.567453293293325);
\draw [line width=0.4pt] (-5.924382412942908,4.567453293293325)-- (-6.06847058678272,4.304321010668247);
\draw [line width=0.4pt] (-1.6385794131290714,-1.1044054587604692)-- (-1.691549904918479,-1.399691975553207);
\draw [line width=0.4pt] (-1.691549904918479,-1.399691975553207)-- (-1.9207902839613062,-1.2061749256162562);
\draw [line width=0.4pt] (-6.6644444158603795,-0.5162495735049628)-- (-6.449345805870232,-0.30712619370177563);
\draw [line width=0.4pt] (-6.449345805870232,-0.30712619370177563)-- (-6.375788951430484,-0.5979687441735583);
\draw [line width=0.4pt] (-4.780057549790742,-2.8715213204799546)-- (-4.9479658249521945,-3.1201310771917847);
\draw [line width=0.4pt] (-4.9479658249521945,-3.1201310771917847)-- (-4.648709322430355,-3.1412390306313154);
\draw [line width=0.4pt] (-6.983924851963651,2.883947939768677)-- (-6.732308948714638,3.047316961410454);
\draw [line width=0.4pt] (-6.732308948714638,3.047316961410454)-- (-6.716635177405956,2.7477266863797527);
\draw [line width=0.4pt] (-7.3173187630953755,-0.39100687134824025)-- (-7.435060774402942,-0.11507786096105876);
\draw [line width=0.4pt] (-7.435060774402942,-0.11507786096105876)-- (-7.615151301385558,-0.3550099390396765);
\draw [line width=0.4pt] (-3.618184424486212,1.0906061274204144)-- (-3.520735721492666,0.8068742993064841);
\draw [line width=0.4pt] (-3.520735721492666,0.8068742993064841)-- (-3.8151790439983024,0.8643471610051935);
\draw [line width=0.4pt] (-4.4911933877523955,-3.357892288908654)-- (-4.22881674685214,-3.2124326783292907);
\draw [line width=0.4pt] (-4.4911933877523955,-3.357892288908654)-- (-4.234033349315948,-3.512387319998221);
\begin{scriptsize}
\draw [fill=black] (-6.961551841554828,3.8026550770671337) circle (2pt);
\draw[color=black] (-6.883875232607593,4.000912552856139) node {$A$};
\draw [fill=black] (-7.82,-3.3) circle (2pt);
\draw[color=black] (-7.743027187202995,-3.103257910411237) node {$B$};
\draw [fill=black] (1.38,-3.46) circle (2pt);
\draw[color=black] (1.4582131007219525,-3.260307192434053) node {$C$};
\draw [fill=black] (-5.131091906063081,-0.9579441228896188) circle (2pt);
\draw[color=black] (-5.0547130066948025,-0.756756873129165) node {$I$};
\draw [fill=black] (-3.165727382224656,-0.25932447791768926) circle (2pt);
\draw[color=black] (-3.0962160779396934,-0.0638923936167421) node {$O$};
\draw[color=wewdxt] (-6.36653642123832,0.5088755761135275) node {$\omega_a$};
\draw[color=wewdxt] (-5.0362366205744715,4.555204136466077) node {$\omega_b$};
\draw[color=wewdxt] (-7.789218152503824,1.6821260947545635) node {$\omega_c$};
\draw [fill=black] (-4.940475892414381,-0.890186688699903) circle (2pt);
\draw[color=black] (-4.86994914549149,-0.6920895217080054) node {$S$};
\end{scriptsize}
\end{tikzpicture}

%% file: corollary43.tex
\definecolor{wewdxt}{rgb}{0.43137254901960786,0.42745098039215684,0.45098039215686275}
\begin{tikzpicture}[scale = 0.5, line cap=round,line join=round,>=triangle 45,x=1cm,y=1cm]
\clip(-8.2758153638788,-7.22053175582545) rectangle (3.7509818358249856,5.376801476007001);
\draw [line width=0.4pt] (-4.45659273715972,1.7731052774903153) circle (2cm);
\draw [line width=0.4pt] (-5.193482396906823,-2.492131357940168) circle (2.3284234920881466cm);
\draw [line width=0.4pt] (0.008781995960095629,-0.7857049628367688) circle (3.146560147631388cm);
\draw [line width=0.4pt] (-3.4897323304900767,-0.3947874887390186) circle (0.37372662533641515cm);
\draw [line width=0.4pt] (-2.349533919721081,-1.045527115741685) circle (5.519145553743315cm);
\draw [line width=0.4pt] (-2.1760108997133054,-4.9790443959366755) circle (1.5818027401105843cm);
\draw [line width=0.4pt] (-6.845210343447558,0.22147205223574362) circle (0.8483435244086578cm);
\draw [line width=0.4pt] (-1.6226408549317501,3.22706198193958) circle (1.1851645748997688cm);
\draw [line width=0.4pt] (-3.780404823840588,-0.22889347700530357) circle (1.876005426460205cm);
\draw [line width=0.4pt] (-7.799658768990033,-1.9156459414118778)-- (-2.1036302650186,4.468137639755203);
\draw [line width=0.4pt] (-3.2722958196382863,-6.486986306020134)-- (-6.796317819118778,2.223582945951254);
\draw [line width=0.4pt] (-3.1075227507638,4.421320291623834)-- (-0.2193375202684087,-6.137013238983451);
\draw [line width=0.4pt] (-4.209787866702343,-6.241720231489525)-- (-0.04013069297161609,3.967219075687097);
\draw [line width=0.4pt] (1.4160099451601242,-5.080579413427568)-- (-7.808194062136377,-0.230666947975938);
\draw [line width=0.4pt] (-7.491424089359551,0.959948669896197)-- (1.6646618353076659,2.7422426876800845);
\draw [line width=0.4pt,dash pattern=on 1pt off 2pt] (-7.142946496263594,1.690192810429242)-- (2.4438786568214326,-3.781247041912613);
\begin{scriptsize}
\draw [fill=black] (-3.4897323304900767,-0.3947874887390186) circle (2pt);
\draw[color=black] (-3.2697373416670725,0.09563039306064772) node {$O_1$};
\draw [fill=black] (-2.4897323304900767,-0.937874887390186) circle (2pt);
\draw[color=black] (-2.1301423447408254,-0.5555667480400703) node {$O_2$};
\draw [fill=black] (-3.780404823840588,-0.22889347700530357) circle (2pt);
\draw[color=black] (-4.05685822876826,0.21772985701703235) node {$S$};
\end{scriptsize}
\end{tikzpicture}

%% file: corollary43fig2.tex
\definecolor{wrwrwr}{rgb}{0.3803921568627451,0.3803921568627451,0.3803921568627451}
\begin{tikzpicture}[scale=0.4, line cap=round,line join=round,>=triangle 45,x=1cm,y=1cm]
\clip(-8.387192819373258,-6.6236933481894615) rectangle (6.691821370290801,9.718157021539078);
\draw [line width=0.4pt] (-8.671004457652302,6.998258543833584) circle (6.096290406295471cm);
\draw [line width=0.4pt] (2.847186969397828,0.3030136229022705) circle (3.392602208410282cm);
\draw [line width=0.4pt] (-2.69,-7.008888774442345) circle (6.008888774442334cm);
\draw [line width=0.4pt] (-0.4469419887122985,-1.1924367564334772) circle (0.22508454277688786cm);
\draw [line width=0.8pt] (-0.9969059372239327,-0.09773506542886885) circle (1cm);
\draw [line width=0.4pt] (-11.982581777322896,13.306027083003876) circle (16.33049085846967cm);
\draw [line width=0.4pt] (26.150620605710056,13.365856698711093) circle (29.302747409941624cm);
\draw [line width=0.4pt] (-9.181688244924084,-15.79740195379269) circle (16.705089715390653cm);
\draw [line width=0.4pt] (-3.7096127764364613,0.7869168558288406) circle (1.8533116579328281cm);
\draw [line width=0.4pt] (-1.9092844173249084,2.68198324263287) circle (1.9256227308258642cm);
\draw [line width=0.4pt,domain=-8.387192819373258:6.691821370290801] plot(\x,{(--23.719962750999315-14.935514813717905*\x)/7.482674632006093});
\draw [line width=0.4pt,domain=-8.387192819373258:6.691821370290801] plot(\x,{(--85.28457844003367--5.829977024915788*\x)/15.654756155501255});
\draw [line width=0.4pt,domain=-8.387192819373258:6.691821370290801] plot(\x,{(--6.190335904581959--1.0646228425383186*\x)/-1.51701756897386});
\draw [line width=0.4pt,domain=-8.387192819373258:6.691821370290801] plot(\x,{(--10.454869529810681--1.8292026130793153*\x)/0.2979629200647895});
\draw [line width=0.4pt,domain=-8.387192819373258:6.691821370290801] plot(\x,{(--21.78870932376293-2.520897835100749*\x)/-16.13474529656231});
\draw [line width=0.4pt,domain=-8.387192819373258:6.691821370290801] plot(\x,{(--4.15433114947465-15.155697112673376*\x)/-6.081922122908548});
\draw [line width=0.4pt] (-7.664501903777584,-0.8031211540057238) circle (13.962329598987795cm);
\draw [line width=0.4pt] (-2.689999999998315,1.2199999999981066) circle (2.219999999998117cm);
\draw [line width=0.4pt,color=wrwrwr] (-2.635948148932724,0.46884145706002456) circle (2.973100762998978cm);
\draw [line width=0.4pt,dash pattern=on 1pt off 2pt,domain=-8.387192819373258:6.691821370290801] plot(\x,{(--11.511088226868184--2.0231211540038303*\x)/4.974501903779268});
\draw [line width=1.2pt,dash pattern=on 1pt off 2pt] (-2.4455030268158557,1.3194364879459985) circle (3.015362165438337cm);
\begin{scriptsize}
\draw[color=wrwrwr] (-8.299669519764926,9.899455285013495) node {$C_3$};
\draw [fill=black] (-2.4455030268158557,1.3194364879459985) circle (2pt);
\draw[color=black] (-2.4606151030376324,1.6972717788605332) node {$S$};
\draw [fill=black] (-7.664501903777584,-0.8031211540057238) circle (2pt);
\draw[color=black] (-7.5244631518054135,-0.4908107113479553) node {$O_2$};
\draw [fill=black] (-2.689999999998315,1.2199999999981066) circle (2pt);
\draw[color=black] (-3.035768214749529,1.509721851128377) node {$O_1$};
\end{scriptsize}
\end{tikzpicture}

%% file: col46fig.tex
\definecolor{wewdxt}{rgb}{0.43137254901960786,0.42745098039215684,0.45098039215686275}
\begin{tikzpicture}[scale = 1, transform shape, line cap=round,line join=round,>=triangle 45,x=1cm,y=1cm]
\clip(-8.414435768545733,-6.592389603788865) rectangle (5.890646608783881,4.51142185132676);
\draw [line width=0.4pt] (-4.45659273715972,1.7731052774903153) circle (2cm);
\draw [line width=0.4pt] (-5.193482396906823,-2.492131357940168) circle (2.3284234920881466cm);
\draw [line width=0.4pt] (0.008781995960095629,-0.7857049628367688) circle (3.146560147631388cm);
\draw [line width=0.4pt] (-3.4897323304900767,-0.3947874887390186) circle (0.37372662533641515cm);
\draw [line width=0.4pt] (-3.186488527237422,0.035271930572883525) circle (0.1524937736885773cm);
\draw [line width=0.4pt] (-2.349533919721081,-1.045527115741685) circle (5.519145553743315cm);
\draw [line width=0.4pt] (-3.29375268090103,-0.8886259657039566) circle (0.15757786775967478cm);
\draw [line width=0.4pt] (-4.003635080100796,-0.32360805557641553) circle (0.14508215230985103cm);
\draw [line width=0.4pt,dash pattern=on 1pt off 1.5pt] (-7.142946496263594,1.690192810429242)-- (2.4438786568214326,-3.781247041912613);
\draw [line width=0.4pt] (-2.1760108997133054,-4.9790443959366755) circle (1.5818027401105843cm);
\draw [line width=0.4pt] (-6.845210343447558,0.22147205223574362) circle (0.8483435244086578cm);
\draw [line width=0.4pt] (-1.6226408549317501,3.22706198193958) circle (1.1851645748997688cm);
\draw [line width=0.4pt,dash pattern=on 1pt off 2pt] (-4.45659273715972,1.7731052774903153)-- (-2.1760108997133054,-4.9790443959366755);
\draw [line width=0.4pt,dash pattern=on 1pt off 2pt] (-5.193482396906823,-2.492131357940168)-- (-1.6226408549317501,3.22706198193958);
\draw [line width=0.4pt,dash pattern=on 1pt off 2pt] (0.008781995960095629,-0.7857049628367688)-- (-6.845210343447558,0.22147205223574362);
\draw [line width=0.4pt,dash pattern=on 1pt off 2pt] (-4.45659273715972,1.7731052774903153)-- (-3.29375268090103,-0.8886259657039566);
\draw [line width=0.4pt,dash pattern=on 1pt off 2pt] (-4.003635080100796,-0.32360805557641553)-- (0.008781995960095629,-0.7857049628367688);
\draw [line width=0.4pt,dash pattern=on 1pt off 2pt] (-5.193482396906823,-2.492131357940168)-- (-3.186488527237422,0.035271930572883525);
\begin{scriptsize}
\draw [fill=black] (-4.45659273715972,1.7731052774903153) circle (2pt);
\draw[color=black] (-4.2963968373017325,2.121890020226862) node {$A_1$};
\draw [fill=black] (-5.193482396906823,-2.492131357940168) circle (2pt);
\draw[color=black] (-5.037536945362169,-2.136032561375289) node {$B_1$};
\draw [fill=black] (0.008781995960095629,-0.7857049628367688) circle (2pt);
\draw[color=black] (0.16497597004246642,-0.4357699605307441) node {$C_1$};
\draw [fill=black] (-3.4897323304900767,-0.3947874887390186) circle (2pt);
\draw[color=black] (-3.3227421855360606,-0.0434016680281569) node {$O_1$};
\draw [fill=black] (-3.186488527237422,0.035271930572883525) circle (2pt);
\draw[color=black] (-3.0320990059045165,0.39256310141916234) node {$B_2$};
\draw [fill=black] (-4.003635080100796,-0.32360805557641553) circle (2pt);
\draw[color=black] (-4.209203883412269,0.17458071669550274) node {$C_2$};
\draw [fill=black] (-3.29375268090103,-0.8886259657039566) circle (2pt);
\draw[color=black] (-3.133824118775557,-0.5374950734017853) node {$A_2$};
\draw [fill=black] (-2.349533919721081,-1.045527115741685) circle (2pt);
\draw[color=black] (-2.189233784973039,-0.6828166632175584) node {$O_2$};
\draw [fill=black] (-6.845210343447558,0.22147205223574362) circle (2pt);
\draw[color=black] (-6.679670910280393,0.5814811681796673) node {$C_3$};
\draw [fill=black] (-2.1760108997133054,-4.9790443959366755) circle (2pt);
\draw[color=black] (-2.014847877194113,-4.621031747225009) node {$A_3$};
\draw [fill=black] (-1.6226408549317501,3.22706198193958) circle (2pt);
\draw[color=black] (-1.4626258358941795,3.5896380773661702) node {$B_3$};
\draw [fill=black] (-3.7804048316818695,-0.22889345378951573) circle (2pt);
\draw [fill=black] (-3.5160627493816867,-0.37976006298002823) circle (2pt);
\end{scriptsize}
\end{tikzpicture}